\documentclass[12pt,a4paper]{article}

\usepackage[utf8]{inputenc}
\usepackage{hyperref}
\usepackage{graphicx}
\usepackage{subfigure} 
\graphicspath{ {imagenes/} }
\usepackage{amsmath}
\usepackage{amsfonts}
\usepackage{amssymb}
\usepackage{color}

\newtheorem{theorem}{Theorem}

\newtheorem{lemma}[theorem]{Lemma}

\newtheorem{remark}[theorem]{Remark}

\newenvironment{proof}[1][Proof]{\textbf{#1.} }{\ \rule{0.5em}{0.5em}}

\usepackage{authblk}

\title{The Dirichlet-to-Neumann map in a disk with a one-step radial potential. An analytical and numerical study}
\author[1]{Juan A. Barcel\'{o}}
\author[1]{Carlos Castro}
\author[1]{Sagrario Lantar\'{o}n}
\author[1]{Susana Merch\'{a}n}
\affil[1]{Departamento de Matem\'{a}tica e Inform\'{a}tica Aplicadas a la Ingeniería Civil y Naval,  Universidad Polit\'{e}cnica de Madrid,  28040 Madrid, Spain.}
\date{}
\begin{document} 
\maketitle
\begin{abstract} 
We consider the Schr\"odinger operator with a potential $q$ on a disk and the map that associates to $q$ the corresponding Dirichlet to Neumann (DtN) map. We give some numerical and analytical results on the range of this map and its stability, for the particular class of one-step radial potentials. 
\end{abstract}

\section{Introduction}

Let $\Omega \subset \mathbb{R}^2$  be a bounded domain with smooth boundary $ \partial  \Omega$. For each $q\in L^\infty(\Omega)$, consider the so called Dirichlet to Neumann map (DtN) given by 
\begin{equation} \label{eq:Lambda}
\begin{array}{rcl}
\Lambda_q : H^{1/2}(\partial \Omega) &\to& H^{-1/2}(\partial \Omega)\\
f& \to &\frac{\partial u}{\partial n}|_{\partial \Omega}.
\end{array}
\end{equation}
where $u$ is the solution of the following problem
\begin{equation}\label{problemi}
\left\{  \begin{array}{ll}
\Delta u + q(x) u=0, \hspace{0.5cm} x \in \Omega,
\\ u= f, \hspace{2.2cm} \partial \Omega,
\end{array}     \right. 
\end{equation}
and $\frac{\partial u}{\partial n}|_{\partial \Omega}$ denotes the normal derivative of $u$ on the boundary $\partial \Omega$.

Note that the uniqueness of $u$ as solution of (\ref{problemi}) requires that $0$ is not a Dirichlet eigenvalue of $\Delta + q$. A sufficient condition to guarantee that $\Lambda_q$ is well-defined is to assume $q(x)<\lambda_1$, the first Dirichlet eigenvalue of the Laplace operator in $\Omega$, since in this case, the solution in (\ref{problemi}) is unique. We assume that this condition holds and let us define the space
$$
L^\infty_{< \lambda_1}(\Omega) = \{q\in L^\infty (\Omega), \mbox{ s. t. } q(x) < \lambda_1, \mbox{ a. e. } \} .
$$

In this work we are interested in the following  map 
\begin{equation} \label{eq_mapl}
\begin{array}{rcl}
\Lambda: L^\infty_{<\lambda_1} (\Omega) & \to & \mathcal{L}(H^{1/2}(\partial \Omega);H^{-1/2}(\partial \Omega))\\
q & \to & \Lambda_q.
\end{array}
\end{equation}
This has an important role in inverse problems where the aim is to recover the potential $q$ from boundary measurements. In practice, these boundary measurements correspond to the associated DtN map and therefore, the mathematical statement of the classical inverse problem consists in the inversion of $\Lambda$. 

It is known that $\Lambda$ is one to one as long as $q\in L^p$ with $p>2$ (see \cite{BIY}, [Bl\.{a}sten, Imanuvilov and Yamamoto 2015]). Therefore, the inverse map $\Lambda^{-1}$ can be defined in the range of $\Lambda$. There are, however, two related important and difficult questions that are not well understood: a characterization of the range of $\Lambda$ and its stability i.e., a quantification of the difference of two potentials, in the $L^\infty$ topology in terms of the distance of their associated DtN maps. Obviously, this stability will affect to the efficiency of any inversion or reconstruction algorithm to recover the potential from the DtN map (see \cite{T1}, [Tejero 2016] and \cite{T2}, [Tejero 2018]). 

The first question, i.e. the characterization of the range of $\Lambda$ is widely open. To our knowledge, the further result is due to Ingerman in \cite{I} [Ingerman 2000], where a difficult characterization is obtained for the adherence with respect to a certain topology.  
Concerning the stability, there are some results when we assume that the potential $q$ has some smoothness. In particular, if $q\in H^{s}(\Omega)$ with $s>0$, the following 
$\log-$stability condition is known (see \cite{BIY}, [Bl\.{a}sten, Imanuvilov and Yamamoto 2015]),
\begin{equation} \label{eq_cal}
\| q_1-q_2\|_{L^\infty} \leq V( \| \Lambda_{q_1}-\Lambda_{q_2} \|_{\mathcal{L}(H^{1/2};H^{-1/2})}) ,
\end{equation}
where $V(t)=C\log(1/t)^{-\alpha}$ for some constants $C,\alpha>0$. 
Stronger stability conditions are known in some particular cases. For example, in \cite{BHQ}, [Beretta, De Hoop and Qiu 2013] it is shown that when $q$ is piecewise constant and all the components where it takes a constant value touch the boundary, the stability is Lipschitz, i.e. there exists a constant $C>0$ such that 
$$
\| q_1-q_2\|_{L^\infty} \leq 
C \| \Lambda_{q_1}-\Lambda_{q_2} \|_{\mathcal{L}(H^{1/2};H^{-1/2})} . 
$$ 

In this work we try to understand better the situation by considering the simplest case of a disk with one-step radial potentials $q$.  More precisely, we give some results on the range of $\Lambda$ and its stability when we restrict to the particular case $\Omega=B(0,1)=\left\{ x \in \mathbb{R}^2: \; r=|x| <1 \right\} $ and  
$
q\in F\subset L^\infty(\Omega)$  
given by
\begin{equation} \label{def:F}
F=\{ q\in L^\infty(\Omega) \; : \; q(r)=\gamma \chi_{(0,b)}(r), \; r=|x|, \; b\in(0,1), \; \gamma \in [0,1] \},
\end{equation}
where $\chi_{(0,b)}(r)$ is the characteristic function of the interval $(0,b)$. Note that $F$ is a two-parametric family depending on $\gamma$ and $b$. 

It is worth mentioning that, as we show below, the solution of (\ref{problemi}) is unique for all $b\in(0,1)$ and $\gamma \geq 0$, and therefore the DtN map is well defined for all these one step potentials. However, we restrict ourselves to the bounded set $F$ to simplify.  

Even in this simple case, a complete analytic answer to the previous questions is unknown. Therefore, we have considered a numerical approach based on a discrete sampling of the set $F$. Given an integer $N>0$ we define $h=1/N$ and
\begin{eqnarray} \nonumber  
F_h&=&\{ q\in L^\infty(\Omega) \; : \; q(r)=\gamma \chi_{(0,b)}(r), \; b=hi, \; \gamma= h_j , \\ && \qquad  i=1,...,N-1, \; j=0,...,N \} . \label{def:Fh}
\end{eqnarray}  
Note that $F_h$ has $N(N-1)+1$ functions from $F$. As $h\to 0$ we obtain a better description of $F$ and, in particular, we should recover the stability properties for $q\in F$. 

Concerning the stability of $\Lambda$, we show that it fails in the sense that inequality (\ref{eq_cal}) does not hold for any continuous function $V(t)$ with $V(0)=0$. The proof relies on the ideas in \cite{A} [Alessandrini 1988] where the analogous result is obtained for the conductivity problem. 

We also obtain some partial stability results when $b$ and $\gamma$ are fixed. To state them we define the subsets $F_b\subset F$, for $b\in(0,1)$, by
\begin{equation} \label{def_Fb}
F_b=\{ q\in L^\infty(\Omega) \; : \; q(r)=\gamma \chi_{(0,b)}(r), \; \gamma \in[0,1] \}, 
\end{equation}
and $G_\gamma\subset F$, for $\gamma\in [0,1]$, by 
\begin{equation} \label{def_Gg} 
G_\gamma=\{ q\in L^\infty(\Omega) \; : \; q(r)=\gamma \chi_{(0,b)}(r), \; b \in(0,1) \}. 
\end{equation}
We prove that, for fixed $b_0>0$, Lipschitz stability holds, if we restrict ourselves to $F_{b_0}$. Therefore, the lack of stability is due to the change of $b_0$ rather than to changes in $\gamma$. Concerning $G_\gamma$, we prove that if $b\geq b_0>0$
there is stability of the DtN map with respect to $b$ for potentials in $G_\gamma$. This suggests a possible stability with respect to the $L^1$-norm of the potentials which is sensitive to the position of the discontinuities, when consider discontinuous functions.  In fact, we show numerical evidences of such stability when considering potentials in $F$.  

For the range of $\Lambda$ we give a characterization in terms of the first two eigenvalues of the DtN map. We also analyze the region where the stability constant is larger and, therefore, the potentials for which any recovering algorithm for $q$ from the DtN map will have more difficulties.  

It is worth mentioning that the results in this paper cannot be easily translated into the closely related, and more classical, conductivity problem where (\ref{problemi}) is replaced by 
\begin{equation}\label{ecuacion_divergencia}
\left\{  \begin{array}{ll}
-\mbox{ div } a(x) \nabla v =0, \hspace{0.5cm} x \in \Omega,
\\ v= f, \hspace{2.2cm} \partial \Omega,
\end{array}     \right. 
\end{equation}
and the Dirichlet to Neumann map, or voltage to current map, is given by
\begin{equation} \label{eq:Lambda_1}
\begin{array}{rcl}
\Lambda_a : H^{1/2}(\partial \Omega) &\to& H^{-1/2}(\partial \Omega)\\
f& \to &a\frac{\partial v}{\partial n}|_{\partial \Omega}.
\end{array}
\end{equation}
We refer to the review paper \cite{Ul} [Uhlmann 2014] and the references therein for theoretical results on the DtN map in this case.  

The rest of this paper is divided as follows: In section 2 below we characterize the DtN map in terms of its eigenvalues  using polar coordinates, in section 3 and 4 we analyze the stability and range results respectively. In section 5 we briefly describe the main conclusions and finally section 5 contains the proofs of the theorems stated in the previous sections.

\section{The Dirichlet to Neumann map}\label{sec1}

In this section we characterize the Dirichlet to Neumann map in the case of a disk. System (\ref{problemi}) in polar coordinates reads
\begin{equation}\label{problema1}
\left\{   \begin{array}{lll}
r^2 \frac{\partial ^2 v}{\partial r^2} +r \frac{\partial v}{\partial r}+ \frac{\partial^2 v}{\partial \theta^2}+r^2q(r)v=0, \hspace{0.3cm}  (r, \theta) \in (0,1) \times [0, 2 \pi), \\  \lim_{r \longrightarrow 0, r>0}v(r, \theta) < \infty, \\ v(1, \theta)=g(\theta), \hspace{0.5cm} \theta \in [0,2 \pi),
\end{array}  \right.
\end{equation}
where $v(r, \theta )= u(r \cos \theta , r \sin \theta) \; \textrm{ and } g(\theta)=f(\cos \theta, \sin \theta)$ is a periodic function.

An orthonormal basis in $L^2(0,2\pi)$ is given by $\{e^{in\theta}\}_{n\in \mathbb{Z}}$. Here we use this complex basis to simplify the notation but in the analysis below we only consider the subspace of real valued functions.  Therefore, any function $g\in L^2(0,2\pi)$ can be written as
\begin{equation} \label{eq:gt}
g(\theta)=\sum_{n \in \mathbb{Z}} g_n  e^{in\theta}, \hspace{0.5cm} g_n=\frac{1}{2\pi}\int_0^{2 \pi} g(t)e^{-int}dt,
\end{equation}
and $\| g \|_{L^2(0,2\pi)}^2 = \sum_{n\in \mathbb{Z}} |g_n|^2$. Associated to this basis we define the usual Hilbert spaces $H^\alpha_{\#}$, for $\alpha >0$, as 
$$
H^\alpha_{\#}=\{ g \; : \; \| g\|_\alpha^2 = \sum_{n\in\mathbb{Z}}(1+n^2)^{\alpha} |g_n|^2 <\infty \}.
$$

The Dirichlet to Neumann map in this case is defined as 
\begin{equation} \label{eq:Lambda_r}
\begin{array}{rcl}
\Lambda_q : H^{1/2}_{\#}(0,2\pi) &\to& H^{-1/2}_{\#}(2\pi)\\
g& \to &\frac{\partial v}{\partial r}(1,\cdot),
\end{array}
\end{equation}
where $v$ is the unique solution of  (\ref{problema1}).

In the above basis the Dirichlet to Neumann map turns out to be diagonal. In fact, we have the following result:

\begin{theorem} \label{th1}
Let $\Omega$ be the unit disk and $q\in F$. Then,
\begin{eqnarray} \label{eq:Dtn}
\Lambda_q\left(e^{in\theta}\right)&=&c_n e^{in\theta}, \quad n\in \mathbb{Z},
\end{eqnarray}
where 
\begin{eqnarray}
\label{c0}
c_0&=&\frac{- b\sqrt{\gamma}J_1(\sqrt{\gamma}b)}{b\log b\sqrt{\gamma}J_1(\sqrt{\gamma}b)+J_0(\sqrt{\gamma}b)}, \\ \label{c0_b}
c_{n}&=&c_{-n}=n\frac{ J_{n-1}(\sqrt{\gamma}b)-b^{2n}J_{n+1}(\sqrt{\gamma}b) }{J_{n-1}(\sqrt{\gamma}b)+ b^{2n}J_{n+1}(\sqrt{\gamma}b)   }, \quad n\in \mathbb{N},
\end{eqnarray}
and $J_n(r)$ are the Bessel functions of first kind. 
\end{theorem}

Note that the range of $\Lambda$, when restricted to $F$, is characterized by the set of sequences $\{ c_n \}_{n\geq 0}$ of the form (\ref{c0})-(\ref{c0_b}) for all possible $b,\gamma$. In particular, when $q=0$ we have 
\begin{equation}\label{potencial_nulo}
c_{n}=n, \hspace{0.4cm} n=0,1,2, \cdot \cdot \cdot, 
\end{equation}
 and this sequence must be in the range of $\Lambda$.

The norm of $\Lambda_q$, when restricted to $F$, is given by 
\begin{equation} \label{eq:norm_l}
\|\Lambda_q \|_{\mathcal{L}(H^{1/2}_{\#}; H^{-1/2}_{\#})}=\sup_{n\geq 0} \frac{|c_n|}{1+n}.
\end{equation}

\begin{proof} (of Theorem \ref{th1})
We first compute $c_0$ in (\ref{eq:Dtn}). As the boundary data at $r=1$  in (\ref{problema1}) is the constant $g(\theta)=1$, we assume that $v(r,\theta)$ is radial, i.e. $v(r,\theta)=a_0(r)$. Then, $a_0$ should satisfy
 \begin{equation}\label{problema_0}
\left\{   \begin{array}{ll}
r^2a_0''+ra_0'+r^2q(r)a_0=0, \hspace{0.5cm} 0<r<1, \\ a_0(1)=1, \hspace{0.5cm} \lim_{r \longrightarrow 0, r>0}a_0(r)< \infty.
\end{array}     \right.
\end{equation}
For $r\in (0,b)$ we solve the ODE with the boundary data at $r=0$, while for $r\in(b,1)$ we use the boundary data at $r=1$.  In the first case the ODE is the Bessel ODE or orden 0 and therefore we have
$$
a_0(r)= \left\{
\begin{array}{ll}
A_0J_0(\sqrt{\gamma} r), & r\in (0,b), \\
1+C_0 \log r, & r\in (b,1),
\end{array}
\right.
$$
where $J_0$ is the Bessel function of the first kind and $A_0$, $C_0$ are constants.  These are computed by imposing continuity of $a_0$ and $a_0'$ at $r=b$. In this way, we obtain
$$
\left\{
\begin{array}{ll}
A_0J_0(\sqrt{\gamma} b) = 1+C_0 \log b \\
A_0\sqrt{\gamma} J_0'(\sqrt{\gamma} b) = C_0 \frac1{b}.
\end{array}
\right.
$$ 
Solving the system for $A_0$ and $C_0$ and taking into account that $\Lambda_q(1)=\frac{\partial v}{\partial r}(1,\theta)=a_0'(1)=C_0$ we easily obtain (\ref{eq:Dtn}).

Similarly, to compute $c_n$ in (\ref{eq:Dtn})  we have to consider $g(\theta)=e^{in \theta}$ in (\ref{problema1}) and therefore we assume that the solution $v(r,\theta)$  can be written in separate variables, i.e. $v(r,\theta)=a_n(r)e^{in\theta}$. Then, $a_n$ must satisfy
\begin{equation}\label{problema_coseno}
\left\{   \begin{array}{ll}
r^2a_n''+ra_n'+\left(r^2q(r)-n^2 \right)a_n=0, \hspace{0.5cm} 0<r<1, \\ a_n(1)=1, \hspace{0.5cm} \lim_{r \longrightarrow 0, r>0}a_n(r)< \infty , \quad n\geq 1.
\end{array}     \right.
\end{equation}
 As in the case of $c_0$, for $r\in (0,b)$ we solve the ODE with the boundary data at $r=0$, while for $r\in(b,1)$ we use the boundary data at $r=1$. We have
$$
a_n(r)= \left\{
\begin{array}{ll}
A_nJ_n(\sqrt{\gamma} r), & r\in (0,b), \\
C_n (r^n-r^{-n})+r^n, & r\in (b,1),
\end{array}
\right.
$$
where $A_n$, $C_n$ are constants.  These are computed by imposing continuity of $a_n$ and $a_n'$ at $r=b$. In this way, we obtain
$$
\left\{
\begin{array}{ll}
A_nJ_n(\sqrt{\gamma} b) = C_n (b^n-b^{-n}) +b^n \\
A_n\sqrt{\gamma} J_n'(\sqrt{\gamma} b) = n C_n (b^{n-1}+b^{-n-1})+nb^{n-1}.
\end{array}
\right.
$$ 
Solving the system for $A_n$ and $C_n$ we obtain in particular
$$
C_n=\frac{-b^nJ'_n(\sqrt{\gamma}b) +n\frac{b^{n-1}}{\sqrt{\gamma}}J_n(\sqrt{\gamma}b) }{-(b^{-n-1}+b^{n-1}) \frac{n}{\sqrt{\gamma}}J_n(\sqrt{\gamma}b) -(b^{-n}-b^n)J'_n(\sqrt{\gamma}b)} .
$$
We simplify this expression using the well known identity
$$
2J'_n(r)=J_{n-1}(r)-J_{n+1}(r),
$$
and we obtain, 
$$
C_n=\frac{- J_{n+1}(\sqrt{\gamma}b) }{b^{-2n}J_{n-1}(\sqrt{\gamma}b) +J_{n+1}(\sqrt{\gamma}b)}. 
$$
Now, taking into account that $\Lambda_q(e^{in\theta})=\frac{\partial v}{\partial r}(1,\theta)=a_n'(1)e^{in \theta}=(2nC_n+n)e^{in\theta}$ we easily obtain (\ref{eq:Dtn}).
\end{proof}
\begin{remark}
In this proof of Theorem \ref{th1} we do not use the restriction $\gamma\leq 1$ that satisfy the potentials in $F$. In fact, the statement of the theorem still holds for any step potential, as in $F$, but with any arbitrary large $\gamma \geq 0$. 
\end{remark}

\section{Stability}

In this section we focus on the stability results for the map $\Lambda$. Some results are analytical and they are stated as theorems. The proofs are given in the appendix below.  We divide this section in three subsections where we consider the negative stability result for $q\in F$ norm  and some partial results when we consider the subsets $F_b$ and $G_\gamma$ defined in (\ref{def_Fb}) and (\ref{def_Gg}).     

\subsection{Stability for $q\in F$}

The first result in this section is the lack of any stability property when $q\in F$. In particular, we prove that inequality  (\ref{eq_cal}) fails, for any continuous function $V(t)$ with $V(0)=0$. 

\begin{theorem} \label{th_st1}
Given $q_0\in F$, there exists a sequence $\{q_k\}_{k\geq 1}\subset F$ such that $\|q_0-q_k\|_{L^{\infty}}=\gamma >0$ for all $k\geq 1$, while 
\begin{equation}\label{C2b}
\| \Lambda_{q_0}-\Lambda_{q_k}\|_{\mathcal{L}(H^{1/2}_{\#}; H^{-1/2}_{\#})} \to 0 , \quad \mbox{ as $k\to \infty$}.
\end{equation}
\end{theorem} 

This result contradicts any possible stability result of the DtN map at $q_0\in F$. Roughly speaking the idea is that the eigenvalues of $\Lambda$, given in Theorem \ref{th1} above, depend continuously on $b$, unlike the $L^\infty$ norm of the potentials. A detailed proof of Theorem \ref{th_st1} is given in the appendix below. 

\subsection{Partial stability}

We give now two partial stability results when we fix $b$ and $\gamma$, respectively. 
\begin{theorem} \label{th 0.1}
Given $b\in(0,1)$ and $q_1,q_2\in F_b$, we have
\begin{equation}\label{estimacion_estabilidad}
 \|q_1-q_2 \|_{L^\infty} \leq \frac{(4,8765)^2}{b^4}     \| \Lambda_{q_1}-\Lambda_{q_2} \|_{\mathcal{L}(H^{1/2}_{\#}; H^{-1/2}_{\#})} .
\end{equation}
On the other hand, given $\gamma \in (0,1]$ and $q_1,q_2\in G_\gamma$, we have
\begin{equation}\label{estimacion_estabilidad_2}
\left| b_1-b_2  \right| \leq \frac{15}{2 \gamma b^3} \| \Lambda_{q_1}-\Lambda_{q_2} \|_{\mathcal{L}(H^{1/2}_{\#}; H^{-1/2}_{\#})},
\end{equation}
where $b=\min \{ b_1,b_2\}$.
\end{theorem}

The proof of this theorem is in the appendix below. 

Inequality (\ref{estimacion_estabilidad}) provides a Lipschitz stability result for $\Lambda$ when $b$ is fixed. This shows that the lack of Lipschitz stability is related to variations in the position of the discontinuity, which is the main idea in the negative result given in Theorem \ref{th_st1}. 

A numerical quantification of this Lipschitz stability for $b$ fixed is easily obtained. We fix $b=b_0$ and consider 
$$
F_{h,b_0}=\{ q\in L^\infty(\Omega) \; : \; q(r)=\gamma \chi_{(0,b)}(r), \; b=b_0, \; \gamma=0+hj, \; j=1,...,1/h-1 \} .
$$ 
and, for $q_0\in F_{h,b_0}$, 
\begin{equation} \label{eq_s2}
C_2(h,q_0,b_0)= \max_{q\in F_{h,b_0}} \frac{\| q_0-q \|_{L^\infty}}{\| \Lambda_{q_0} - \Lambda_{q} \|_{\mathcal{L}(H^{1/2}_{\#}; H^{-1/2}_{\#})}}, 
\end{equation}
then $C_2(h,q_0,b_0)$ remains bounded as $h\to 0$ for all $q_0\in F_h$. In Figure \ref{fig:estabilidad_b_iguales} we show the behavior of $C_2(h,q_0,b_0)$ when $h=10^{-4}$ for different values of $b_0$. To illustrate the behavior with respect to $b\to 0$ we plot in the left hand side of Figure \ref{fig:estabilidad_b_iguales} the graphs of the functions 
\begin{equation} \label{eq_s1}
C_{2,\min}(b)=\min_{q\in F_{h,b}}C_2(10^{-4},q,b), \mbox{ and } C_{2,\max}(b)=\max_{q\in F_b}C_2(10^{-4},q,b).
\end{equation}
We see that both constants become larger for small values of $b$. We also see that both graphs are close in this logarithmic scale. However, the range of the interval $[C_{2,\max}(b),C_{2,\min}(b)]$ is not small, as showed in the right hand side of Figure \ref{fig:estabilidad_b_iguales}.

\begin{figure}[htbp]
  \centering
 \subfigure[]{  \includegraphics[width=0.45\textwidth]{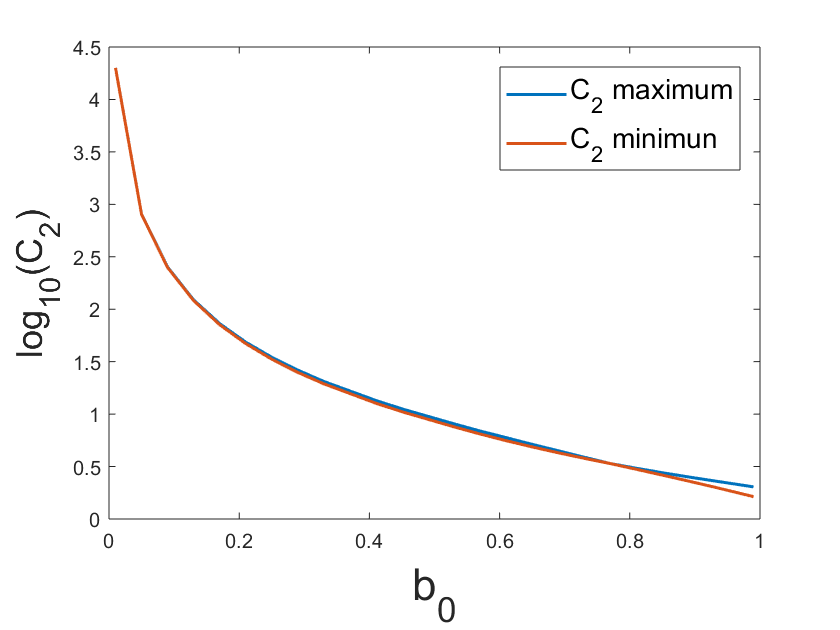}}
  \subfigure[]{  \includegraphics[width=0.45\textwidth]{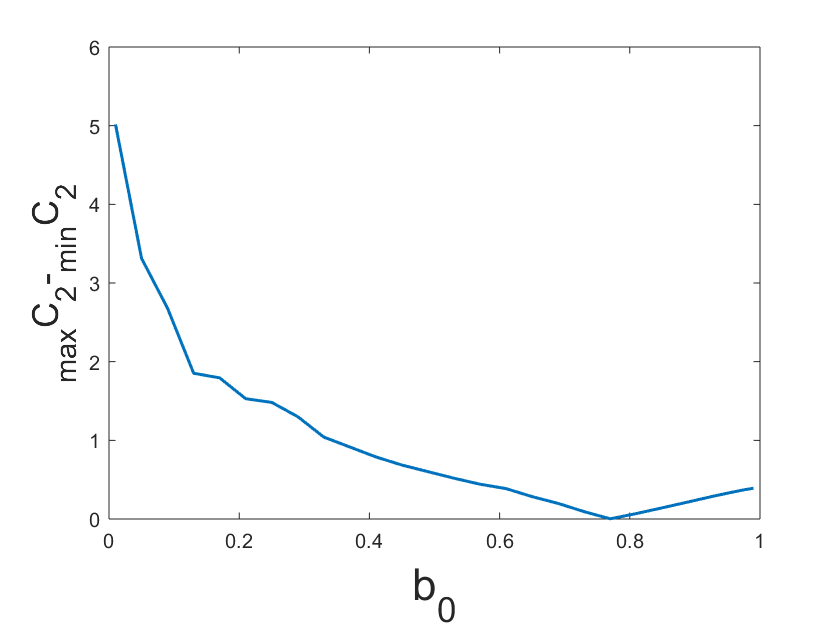}}
  \caption{Numerical estimate of the stability constant $C_2$ in (\ref{eq_s2}) for $h=10^{-4}$. To illustrate the behavior on $b$ we plot the maximum and minimum value when $q\in F_{h,b}$ with respect to $b$  in logarithmic scale (left), and its range in normal scale (right). }\label{fig:estabilidad_b_iguales}
\end{figure}

Concerning inequality (\ref{estimacion_estabilidad_2}) in Theorem \ref{th 0.1}, it provides a stability result of  $\Lambda$ with respect to the position of the discontinuity. In particular, this means that we can expect some Lipschitz stability if we consider a norm for the potentials that is sensitive to the position of the discontinuity. This is not the case for the $L^\infty$ norm but it is true for the $L^p$-norm for some $1\leq p <\infty$. In particular,
\begin{equation} \label{eq_s2b}
C_2(h,q_0,b_0)= \max_{q\in F_{h,b_0}} \frac{\| q_0-q \|_{L^1}}{\| \Lambda_{q_0} - \Lambda_{q} \|_{\mathcal{L}(H^{1/2}_{\#}; H^{-1/2}_{\#})}}, 
\end{equation}
is bounded as $h\to 0$ and $b\geq b_0 >0$. In Figure \ref{fig:estabilidad_norma1} we show the values when $h=10^{-4}$. We observe that the constant blows up as $b\to 0$.

\begin{figure}[htbp]
  \centering
\includegraphics[width=0.50\textwidth]{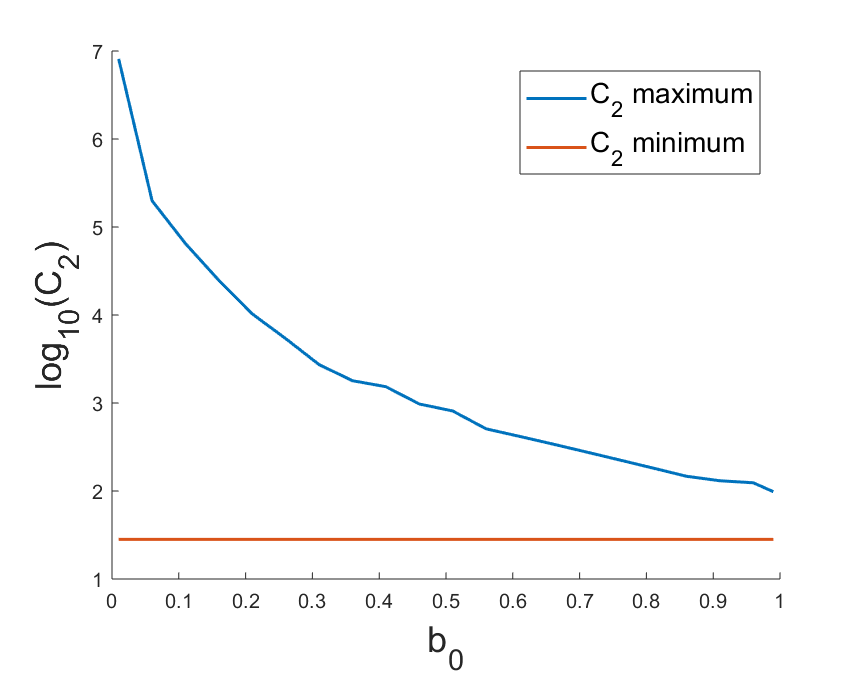}
 \caption{ $C_2(h,q)$ for $b>b_0$ when $h=10^{-4}$.} 
 \label{fig:estabilidad_norma1}
 \end {figure}

\section{Range of the DtN map}

In this section we are interested in the range of $\Lambda$ when $q\in F$, i.e. the set of sequences $\{ c_n \}_{n\geq 0}$ of the form (\ref{c0})-(\ref{c0_b}) for all possible $b,\gamma \in[0,1]\times[0,1]$. 

As $F$ is a bi-parametric family of potentials, it is natural to check if we can characterize the family $\{ c_n \}_{n\geq 0}$ with only the first two coefficients $c_0$ and $c_1$. In this section we give numerical evidences of the following facts:
\begin{enumerate}
\item  The first two coefficients $c_0$ and $c_1$ in (\ref{c0})-(\ref{c0_b}) are the most sensitive with respect to $(b,\gamma)$ and, therefore, the more relevant ones to identify $b$ and $\gamma$ from the DtN map. 
\item The function
\begin{eqnarray} \label{DtNh}
\Lambda^h:F_h &\to& \mathbb{R}^2 \\ \nonumber
q &\to & (c_0,c_1),
\end{eqnarray}
is injective. This means, in particular, that the DtN map  can be characterized by the coefficients $c_0$ and $c_1$, when restricted to functions in $F_h$.  We also illustrate the set of possible coefficients $c_0,c_1$. 
\item The lack of stability for $\Lambda$ is associated to  higher density of points in the range of $\Lambda^h$. This occurs when either $b$ or $\gamma$ are close to zero. 
\end{enumerate}

\subsection{Sensitivity of $c_n$}

To analyze the relevance and sensitivity of the coefficients $c_n=c_n(b,\gamma)$ to identify the parameters $(b,\gamma)$ we have computed their range when $(b,\gamma)\in [0,1]\times [0,1]$, and the norm of their gradients. As we see in (\ref{table_1}) the range decreases for large  $n$. This means that, for larger values of $n$, the variability of $c_n$ is smaller and they are likely to be less relevant to identify $q$. 

However, even if the range of $c_n$ becomes smaller for large $n$ they could be more sensitive to small perturbations in $(b,\gamma)$ and this would  make them useful to distinguish different potentials. But this is not the case.  In Figure \ref{fig:gradientegammaconstante} we show that for the given values of $\gamma=0.1,0.34,0.67,0.99$ and $b\in (0,1]$ the gradients of the first two coefficients, with respect to $(b,\gamma)$, are larger than the others. Therefore, we conclude that the two first coefficients $c_0$ and $c_1$ are the most sensitive, and therefore relevant, to identify the potential $q$. 

We also see in Figure \ref{fig:gradientegammaconstante} that these gradients are very small for $b<<1$. This means, in particular, that identifying potentials with small $b$ from the DtN map should be more difficult.
\begin{table} \label{table_1}
$$ 
\begin{array}{|l|l|}
\mbox{Coefficient} & \mbox{Range} \\ 
\hline
c_0  & 0.5523 \\
\hline
c_1  & 0.2486 \\
\hline
c_2  & 0.1588 \\
\hline
c_3  & 0.1157 \\
\hline
c_4  & 0.0904 \\
\hline
c_5  & 0.0736 \\
\end{array}
$$
\caption{Range of the coefficients, i.e. for each $c_n$ the range is defined as $\max_{q\in F_h} c_n - \min_{q\in F_h} c_n$.}
\end{table}

\begin{figure}[htbp]
  \centering
\includegraphics[width=0.99\textwidth]{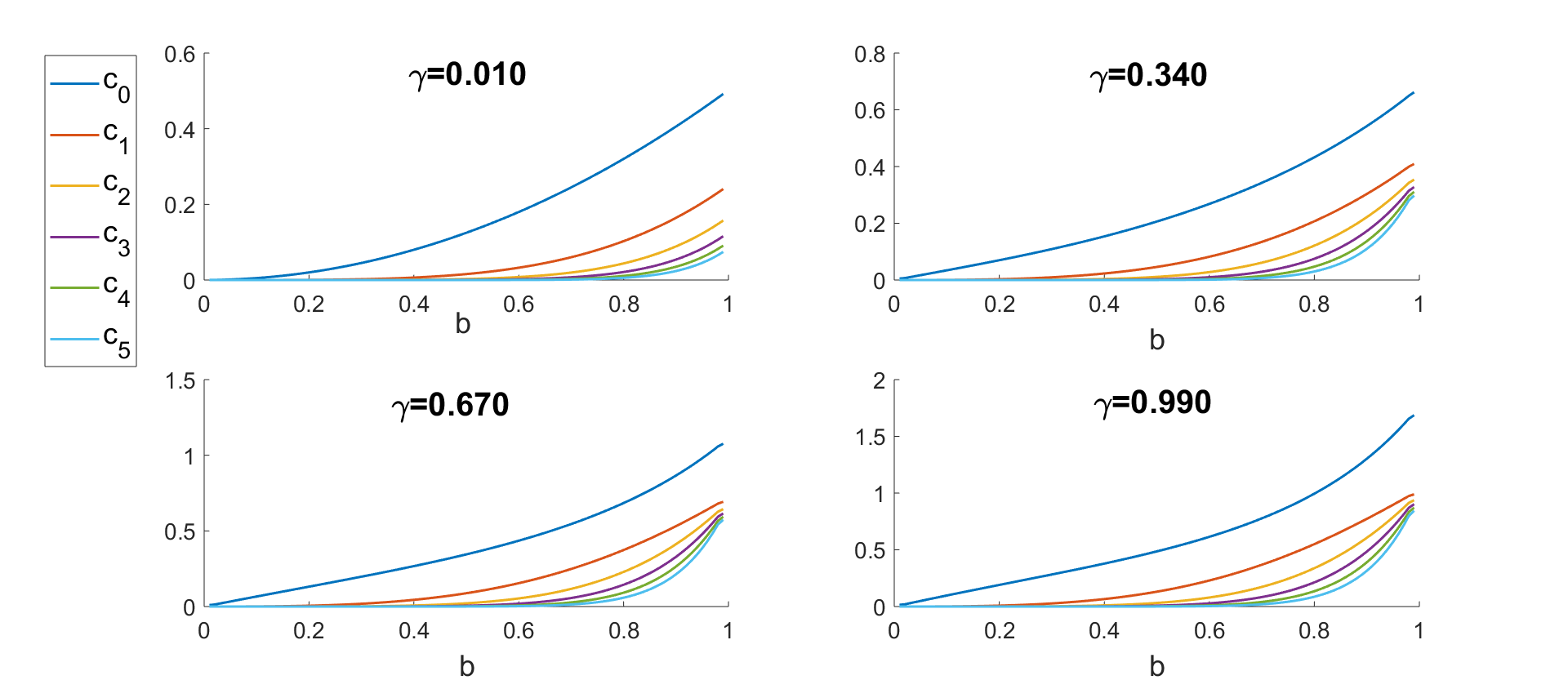}

 \caption{Norm of the gradient of the coefficients $c_n(\gamma,b)$ in terms of $b\in(0,1)$ for different values of $\gamma$. We see that the gradients of higher coefficients $n\geq 2$ are smaller than those of the first two ones. We also observe that these gradients become small for small values of $b$.   } \label{fig:gradientegammaconstante}
 \end {figure}

\subsection{Range of the DtN in terms of $c_0,c_1$}

Now we focus on the range of the DtN in terms of the relevant coefficients $(c_0,c_1)$, i.e. the range of the map $\Lambda^h$ in 
(\ref{DtNh}): $R(\Lambda^h)$. In Figure \ref{fig:graficoc0c1} we show this range. 
 
\begin{figure}[htbp]
  \centering
\includegraphics[width=0.50\textwidth]{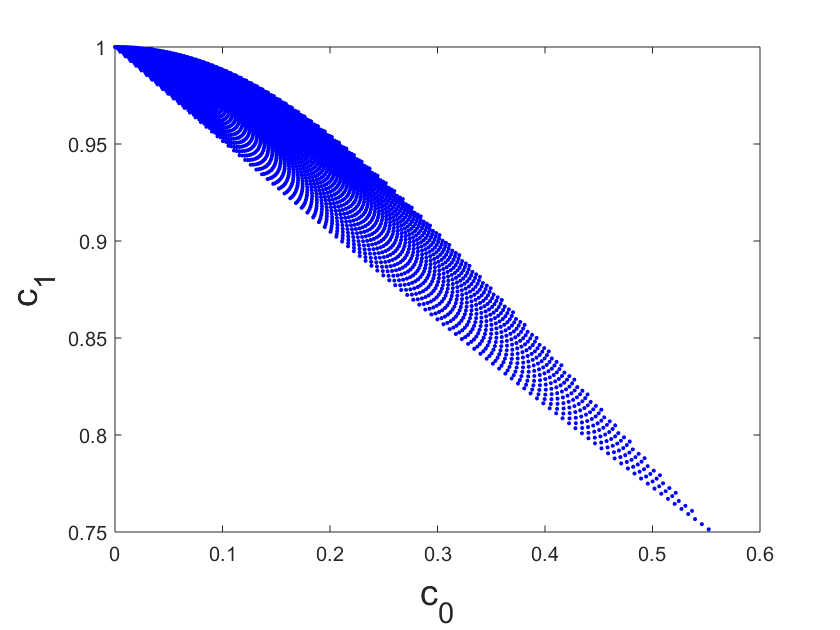}
 \caption{Range of the discrete DtN map in \ref{DtNh}} \label{fig:graficoc0c1}
 \end {figure}
Coordinates lines for fixed $\gamma$ and $b$ are given in Figure \ref{fig:curvasbgammaconstante}. We observe that $R(\Lambda^h)$ is a convex set between the curves 
\begin{eqnarray*}
&& r_{low}: \{ (c_0(\gamma,1), c_1(\gamma,1)),  \mbox{ with } \gamma\in[0,1]\},  \\
&& r_{up}: \{ (c_0(1,b), c_1(1,b)),  \mbox{ with } b\in[0,1]\}.
\end{eqnarray*}

Note also that, in the $c_0,c_1$ plane, the length of the coordinates lines associated to  $b$ constant are segments that become smaller as $b\to 0$. Analogously, the length of those associated to constant $\gamma$ become smaller as $\gamma\to 0$. Thus, the region where either $b$ or $\gamma$ are small produces the higher density of points in the range of $\Lambda^h$. This corresponds to the upper left part of its range (see Figure \ref{fig:graficoc0c1}). On the other hand, this Figure provides a numerical evidence of the injectivity of $\Lambda^h$ too. In fact, any point inside  $R(\Lambda^h)$ is the intersection of two coordinates lines associated to some unique $b_0$ and $\gamma_0$.  

\begin{figure}[htbp]
  \centering 
 \subfigure[]{  \includegraphics[width=0.4\textwidth]{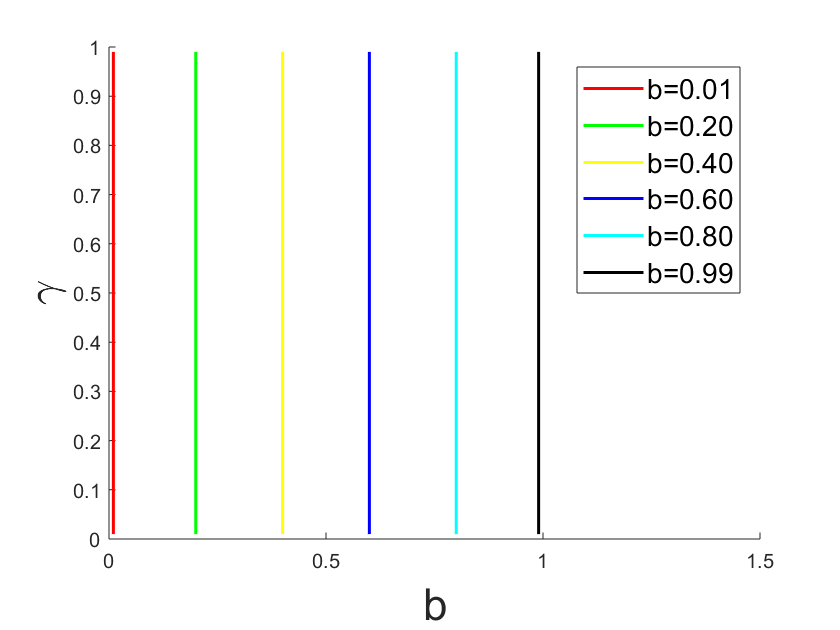}}
  \subfigure[]{  \includegraphics[width=0.4\textwidth]{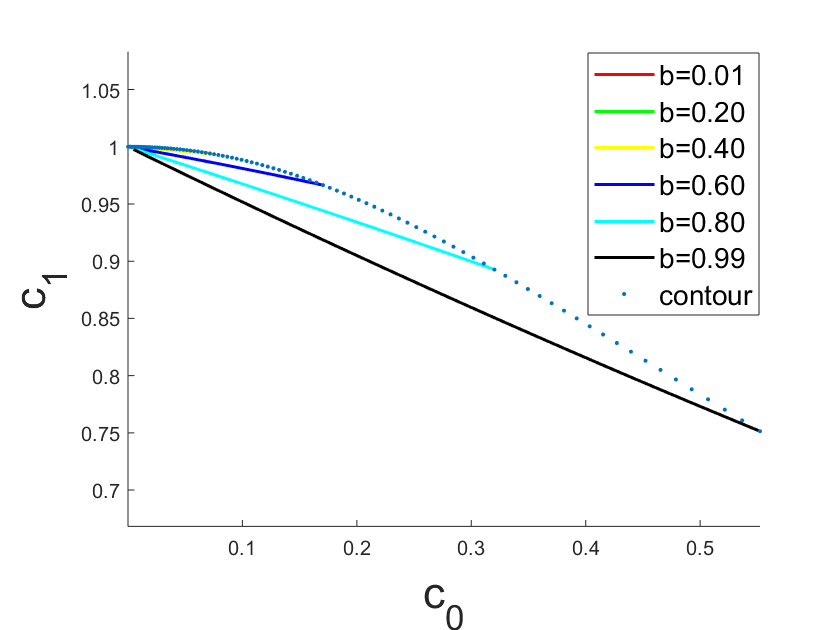}}
  \subfigure[]{  \includegraphics[width=0.4\textwidth]{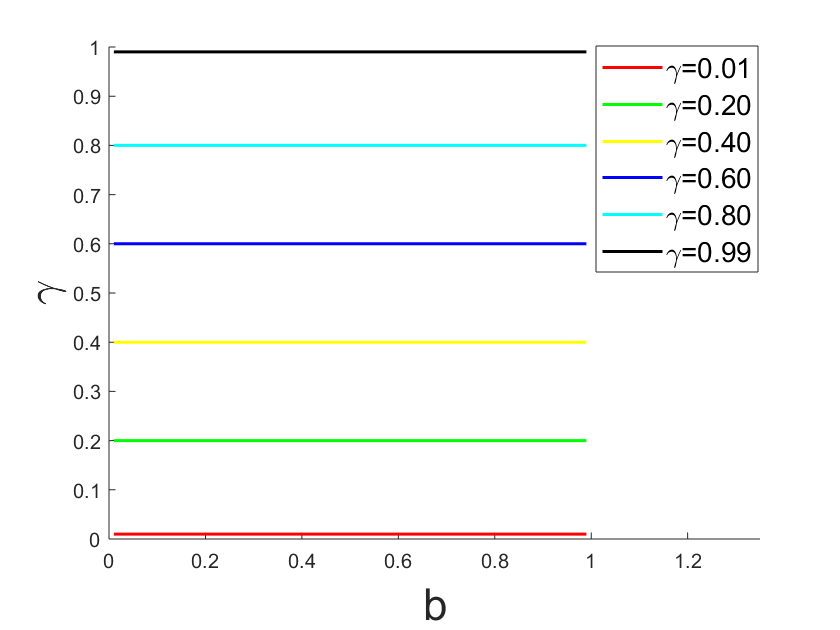}}
  \subfigure[]{  \includegraphics[width=0.4\textwidth]{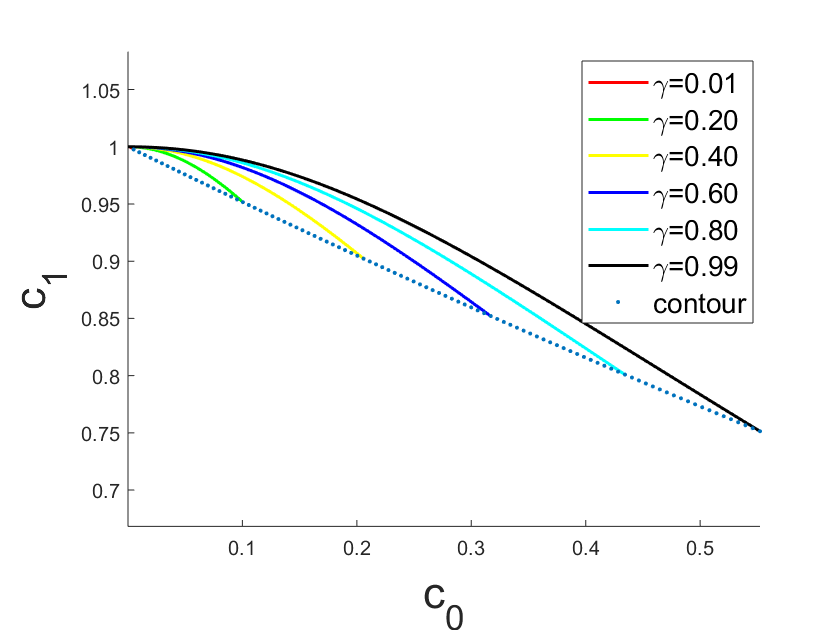}}
  \caption{Coordinates lines of the map $\Lambda^h$ defined in (\ref{DtNh}). The upper figure contains the coordinates lines associated to  $b$ constant, while the lower one those corresponding to $\gamma$ constant. }\label{fig:curvasbgammaconstante}
\end{figure}

The higher density of points in the upper left hand side of the range of $\Lambda^h$ should correspond to potentials $q$ with large stability constant $C_2(h,q)$. In Figure \ref{fig:curvanivelC2_1} we show the level sets of $C_2(h,q)$ for $h=10^{-4}$ and different $q\in F_h$. The region with larger constant corresponds to small values of $b$ (upper right  figure) and larger values of $c_1$ (upper left and lower figures). On the other hand, the region with lower stability constant is for $b$ close to $b=1$, which corresponds to the lower part of the range of $\Lambda^h$ when $c_0$ is small.   

\begin{figure}[htbp]
  \begin{center}
	\begin{tabular}{cc}
\includegraphics[width=0.50\textwidth]{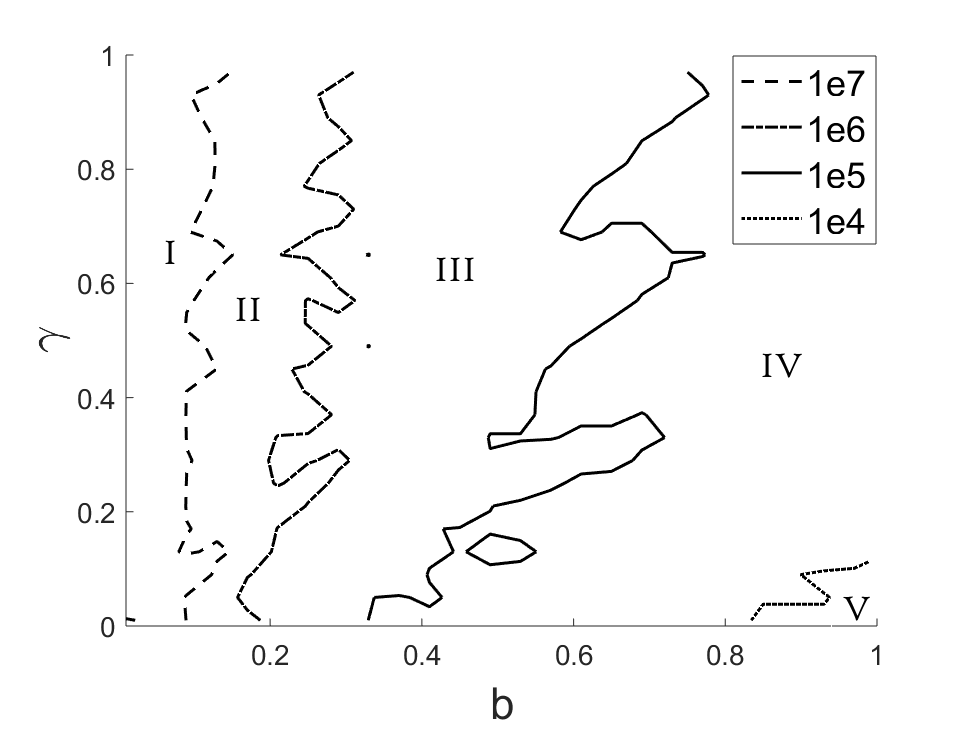} &
\includegraphics[width=0.50\textwidth]{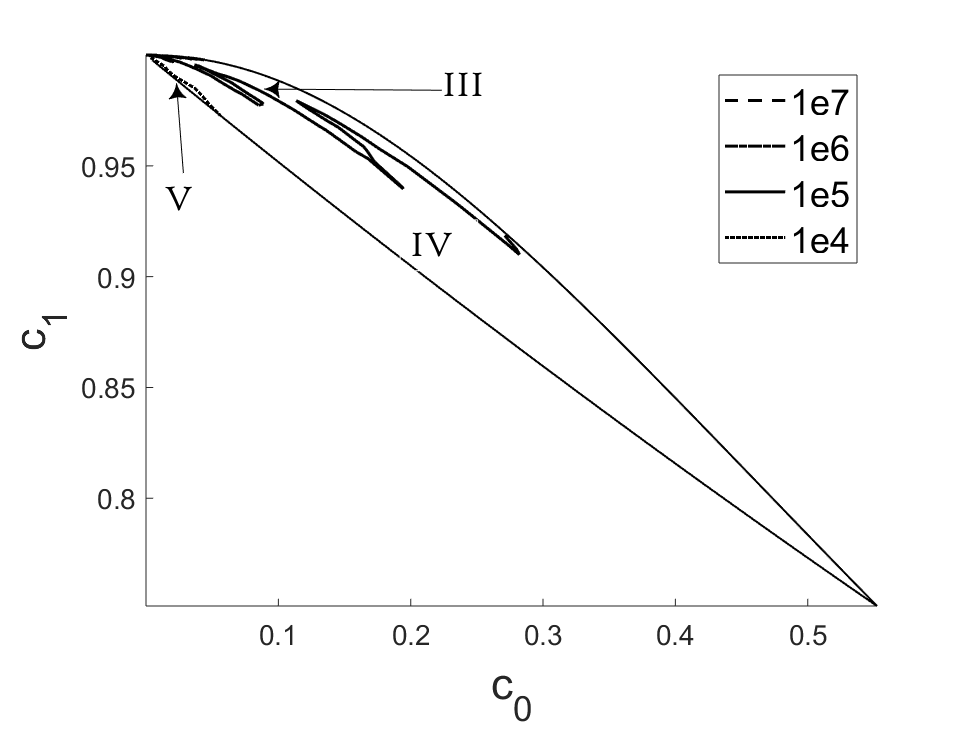} 
\end{tabular}

\includegraphics[width=0.50\textwidth]{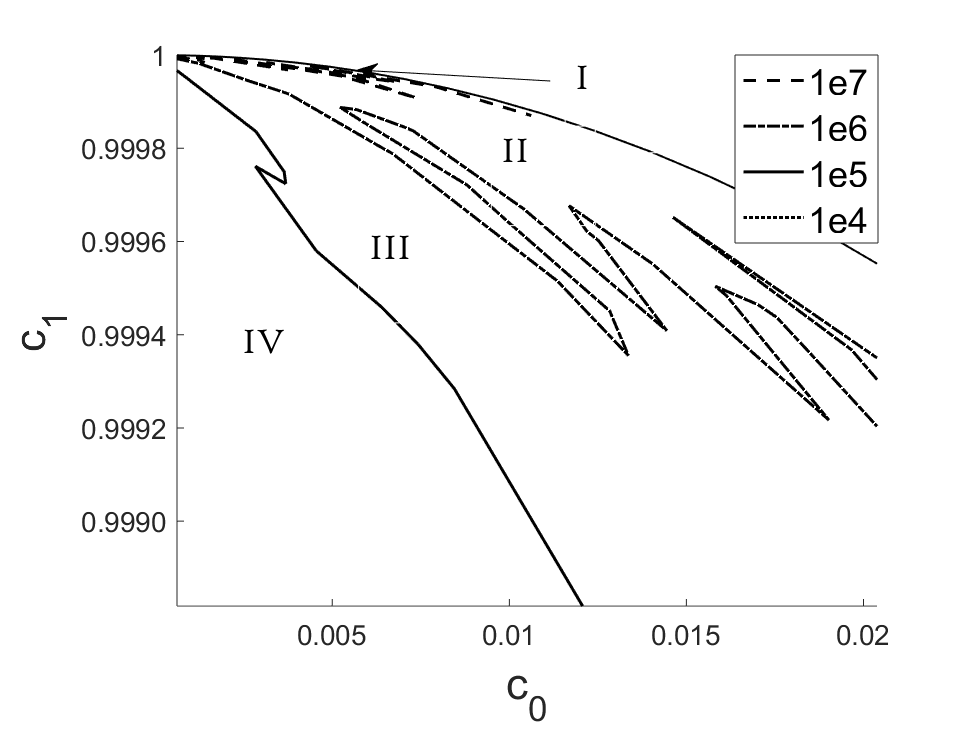}
\end{center}
 \caption{Level sets of the $C_2(b,\gamma)$ for $q\in F_h$ and $h=10^{-4}$ in terms of $(b,\gamma)$ (upper left) and in terms of $ (c_0,c_1)$ (upper right). A zoom of the upper left region in this last figure is in the lower figure. Regions separated by the level sets are indicated: region I corresponds to the potentials with stability constant larger that $10^7$, region II corresponds to those with stability constant lower that $10^7$ but larger than $10^6$, and so on. } \label{fig:curvanivelC2_1}
 \end {figure}

It is interesting to analyze the set of potentials with the same coefficient $c_0$ or $c_1$. We give in Figure \ref{fig:curvasc0c1constante} the coordinates lines of the inverse map $(\Lambda^h)^{-1}$. When increasing the value of either $c_0$ (light lines) or $c_1$ (dark lines) we obtain lines closer to the left part of the $(b,\gamma)$ region. We see that the angle between coordinate lines  becomes  very small for $b$ small. In this region, close points could be the intersection of coordinates lines associated to not so close parameters $(b,\gamma)$. This agrees with the region where the stability constant is larger. 

\begin{figure}[htbp]
  \centering 
 \subfigure[]{  \includegraphics[width=0.45\textwidth]{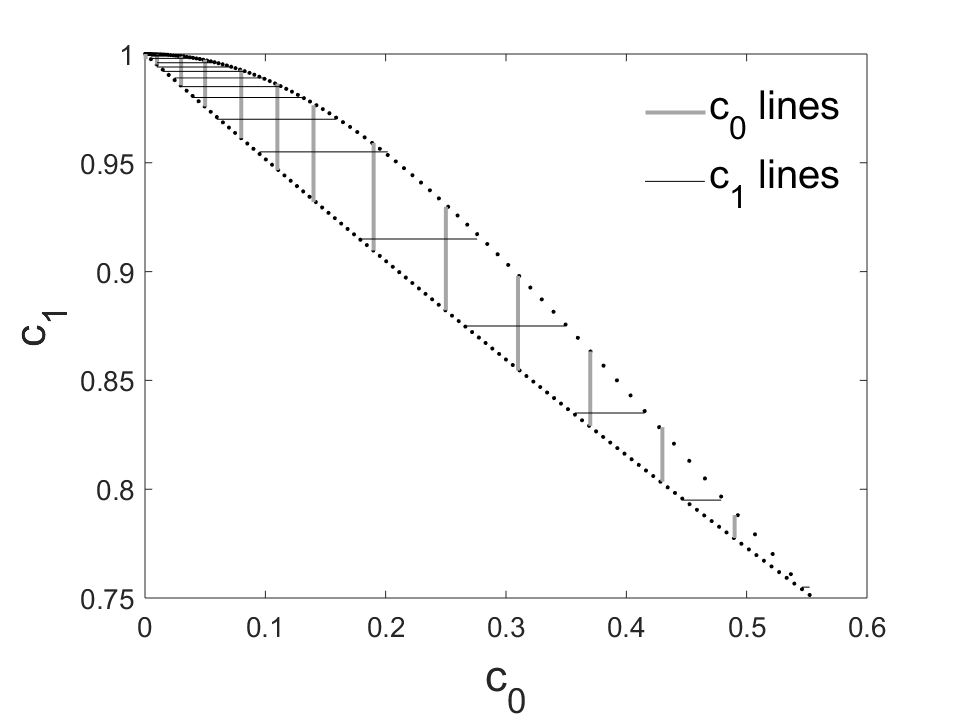}}
  \subfigure[]{  \includegraphics[width=0.45\textwidth]{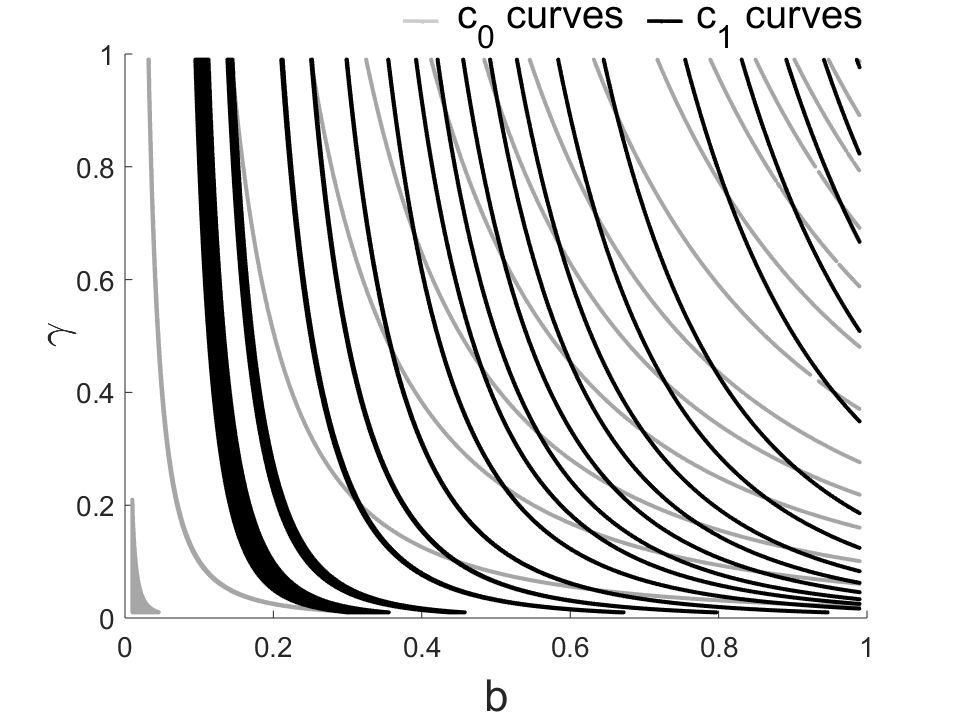}}
  
  \caption{Coordinates lines of the map $(\Lambda^h)^{-1}$ defined in (\ref{DtNh})}\label{fig:curvasc0c1constante}
\end{figure}

\section{ Conclusions}

We have considered the relation between the potential in the Schr\"odinger equation and the associated DtN map in one of the simplest situations, 
i.e. for the subset of radial one step potentials in dimension 2. In particular we have focused on two difficult problems: the stability of the map $\Lambda$ (defined in \ref{eq_mapl}) and its range.  
In this case, the map $\Lambda$ is easily characterized in terms of the Bessel functions and this allows us to give some analytical and numerical results on these problems. 
We have proved the lack of any possible stability result, by adapting the argument in \cite {A} [Alessandrini 1988] for the conductivity problem. We have also obtained some partial Lipschitz stability when the position of the discontinuity is fixed in the potential and numerical evidences of the stability with respect to the $L^1$ norm. Finally, we have characterized numerically the range of $\Lambda$ in terms of the first two eigenvalues of the DtN map and given some insight in the regions where stability of $\Lambda$ is worse.

\section*{Acknowledgements} 
The first three authors have been partially supported  by project MTM2017-85934-C3-3-P from the MICINN (Spain).
The fourth author has been partially supported by project MTM2016-80474-P of MINECO, Spain.


\section*{Appendix}

To prove Theorems \ref{th_st1} and \ref{th 0.1} we will need the following technical results about the Bessel functions.

\begin{lemma}\label{Lema_Bessel_cero} Let 
 $J_\mu (r)$ the Bessel functions of first kind of order $\mu > -\frac{1}{2}$. It is well known (see \cite{Gr} [Grafacos 2008]) that 
$$
	J_\mu(r)=\frac{r^\mu}{2^\mu \Gamma \left(  \mu +1 \right)}+S_\mu(r),
$$
	where 
$$
	S_\mu (r)=\frac{r^\mu}{2^\mu \Gamma \left( \mu+\frac{1}{2}  \right)\Gamma \left( \frac{1}{2} \right)}\int_{-1}^1 \left(   \cos rt -1\right)\left( 1-t^2  \right)^{\mu-\frac{1}{2}}dt.
$$
For $n=0, 1,2, \cdot \cdot \cdot$ and $r \in (0,1) $ the following holds:	
\begin{equation}\label{estimacion_resto_bessel}
-\frac{r^{n +2}}{2^{n+1} \Gamma \left(n+ \frac{3}{2}  \right) \sqrt{\pi}  }  \int_0^1  \left(1-t^2   \right)^{n +\frac{1}{2}}dt \leq S_n (r)  
\end{equation} 
$$\leq   -\frac{r^{n +2}\cos r}{2^{n+1} \Gamma \left(n+ \frac{3}{2}  \right) \sqrt{\pi}  }  \int_0^1  \left(1-t^2   \right)^{n +\frac{1}{2}}dt,$$

\begin{equation}\label{estimacion_inferior}
0 <\frac{r^n}{2^{n+1} n!} \leq J_n(r) \leq  \frac{r^n}{2^{n} n!}, 
\end{equation}
and
\begin{equation}\label{estimacion_derivada_bessel}
0< \frac{r^{n}}{2^{n+2} n!} \leq J_{n+1}'(r) \leq  \frac{r^{n}}{2^{n+1}n!}. 
\end{equation} 
 More explicit estimates for $S_0(r)$ and $S_2(r)$ are given by,
\begin{equation}\label{bessel_orden_0}
 -\frac{r^2}{4} \leq S_0(r) \leq -\frac{r^2\cos r}{4} \leq 0, 
\end{equation}
\begin{equation}\label{bessel_orden_2}
 -\frac{r^4}{15 \pi} 0.4909 \leq S_2(r) \leq -\frac{r^4\cos r}{15 \pi}0.4909.
\end{equation}

\end{lemma}

\begin{proof} 
To prove (\ref{estimacion_resto_bessel}) we use 
\begin{equation}\label{papel_0}
   \frac{r^2t^2}{2}\cos r              \leq 1- \cos (rt) \leq \frac{r^2t^2}{2}, \;\; \;\; \;   r,t \in (0,1), 
  \end{equation} 
and $$\int_0^1 t^2 \left(1-t^2   \right)^{n -\frac{1}{2}}dt= \frac{1}{2 \left( n +\frac{1}{2}  \right)}\int_0^1  \left(1-t^2   \right)^{n +\frac{1}{2}}dt. $$
From  (\ref{estimacion_resto_bessel}) and the well-known identities,
$$\begin{array}{lll}
\Gamma \left( \frac{1}{2} \right) = \sqrt{\pi}, \hspace{0.25cm} \\ \Gamma (r+1)=r \Gamma (r), \hspace{0.3cm} r>0, \hspace{0.25cm} \\ 2 J'_{n+1}(r)=J_{n}(r)-J_{n+2}(r), \hspace{0,3cm} r>0,
\end{array} $$ (see \cite{LE} [Lebedev 1972]),
 we get (\ref{estimacion_inferior}), (\ref{estimacion_derivada_bessel}), (\ref{bessel_orden_0}) and (\ref{bessel_orden_2}).

\end{proof}

The following lemma will be used in the proof of Theorem \ref{th 0.1}.

\begin{lemma} \label{a}
For $0 < r \leq  s <1$ and $n=0,2$,  we have
$$\int_0^1 \left( 1-\cos (rt)  \right) \left(1-t^2    \right)^{n-\frac{1}{2}} dt \leq \frac{\pi r^2}{28n+8} ,$$ and $$\int_0^1 \left( \cos (rt)- \cos (st)  \right) \left(1-t^2    \right)^{n-\frac{1}{2}} dt \leq  \frac{\pi(s^2-r^2)}{28n+8} .$$
\end{lemma}
\begin{proof} The previous estimates are consequence of (\ref{papel_0}) and the inequality
$$ \cos r- \cos s    = 2 \sin  \frac{s+r}{2}      \sin  \frac{s-r}{2}  \      \leq \frac{s^2-r^2}{2}  . $$

\end{proof}

\begin{proof} (of Theorem \ref{th_st1})

We take $\gamma =1$ without loss of generality.
For $b_0 \in (0,1)$  we will consider the fixed potential 
$$  q_{0}(r, \theta)=\left\{   \begin{array}{ll}
	1, \;\; 0<r<b_0, \\ 0, \;\; b_0 \leq r <1,
	\end{array}    \right.  $$ and a positive integer $k(b_0)$ satisfying $b_0+\frac{1}{k(b_0)}<1$.
  We define the potentials
\begin{equation}\label{b_k}
      q_k(r, \theta)=\left\{   \begin{array}{ll}
	1, \;\; 0<r<b_k, \\ 0, \;\; b_k \leq r <1,
	\end{array}    \right.   \;\; k=1,2, \cdot \cdot \cdot,
\end{equation}	
	 with $b_k=b_0+\frac{1}{k(b_0)+k}$.
	
	We have that $\|q_0-q_k \|_{L^\infty}=1$ and to have (\ref{C2b}) we will prove for $g \in H^{1/2}_{\#}$ that
\begin{equation}\label{C2bd}
\| \left(\Lambda_{q_0}-\Lambda_{q_k}\right)g\|^2_{ H^{-1/2}_{\#}}  \leq C |b_0-b_k|^2 \| g\|^2_{H^{1/2}_{\#}}\leq \frac{C}{k^2}  \| g\|^2_{ H^{1/2}_{\#}},
\end{equation}	
where $C$ is a constant independent of $k$ and $g$.

If $g(\theta ) = \sum_{n \in Z}g_ne^{in\theta}$, by (\ref{c0}) and (\ref{c0_b}) we have
$$\| \left(\Lambda_{q_0}-\Lambda_{q_k}\right)g\|^2_{ H^{-1/2}_{\#}}   \leq     \left| \frac{b_kJ_1(b_k)}{b_kJ_1(b_k)\log b_k+J_0(b_k)} - \frac{b_0J_1(b_0)}{b_0J_1(b_0)\log b_0+J_0(b_0)}\right|^2  \left| g_0
	\right|^2$$
	$$+ \sum_{n=1}^\infty  \left|   \frac{J_{n-1}(b_k)-b_k^{2n}J_{n+1}(b_k)}{J_{n-1}(b_k)+b_k^{2n}J_{n+1}(b_k)}  -  \frac{J_{n-1}(b_0)-b_0^{2n}J_{n+1}(b_0)}{J_{n-1}(b_0)+b_0^{2n}J_{n+1}(b_0)}    \right|^2 (1+n^2)^{1/2}\left(  \left|g_n   \right|^2 +\left|g_{-n}  \right|^2  \right) 
	$$ $$=I_0^2\left| g_0\right|^2+ \sum_{n=1}^\infty  I_n^2 (1+n^2)^{1/2}  \left(  \left|g_n   \right|^2 +\left|g_{-n}  \right|^2  \right)      .$$
We start by estimating $I_0$.

From (\ref{estimacion_inferior}), (\ref{estimacion_resto_bessel}) and (\ref{bessel_orden_0}) $J_1(r) \leq \frac{r}{2}$ when $r \in (0,1)$ and  $$   rJ_1(r) \log r +J_0(r) \geq  \frac{r^2 \log r}{2}+1- \frac{r^2}{4}, \hspace{0.5cm} r \in (0,1).$$
Since $\frac{r^2 \log r}{2}+1- \frac{r^2}{4}$  it is a decreasing function in $(0,1)$ we have

\begin{equation}\label{decreciente}
 rJ_1(r) \log r +J_0(r) \geq \frac{3}{4}, \hspace{0.5cm} r \in (0,1).
 \end{equation} 
A simple calculation and this inequality gives us 
$$I_0 \lesssim b_kb_0J_1(b_k)J_1(b_0) \left| \log b_k - \log b_0 \right| +J_1 (b_k)J_0(b_0 ) \left| b_k-b_0 \right|$$ $$+b_0J_0(b_0) \left| J_1(b_k)-J_1(b_0)  \right|+b_0J_1(b_k) \left| J_0(b_k)-J_0(b_0)  \right|,$$
where the symbol $\lesssim$ denotes that the left hand side is bounded by a constant times the right hand one. Thus, combining the mean value theorem, the identity $J'_0(r)=-J_1(r)$, the fact that $b_k, b_0 \in (0,1)$  and (\ref{estimacion_inferior}), we easily get
\begin{equation}\label{I_0}
I_0 \lesssim \left| b_k-b_0\right|.
\end{equation}
Now we deal with $I_k, \; k=1,2, \cdot \cdot \cdot$. We will  use  the mean value Theorem, $b_k, b_0 \in (0,1)$, $\left| b_k^{2n}-b_0^{2n} \right| \lesssim \frac{\left| b_k-b_0\right|}{n}$, (\ref{estimacion_inferior}) and (\ref{estimacion_derivada_bessel}) to obtain
$$I_n \lesssim \frac{J_{n+1}(b_k)J_{n-1}(b_0)\left|b_k^{2n}-b_0^{2n}\right|
 +b_0^{2n}J_{n-1}(b_0) \left| J_{n+1}(b_k)-J_{n+1}(b_0)   \right|         }{J_{n-1}(b_k)J_{n-1}(b_0)} $$
$$+ \frac{  b_k^{2n}
 J_{n+1}(b_0) \left| J_{n-1}(b_k)-J_{n-1}(b_0)   \right|          }{J_{n-1}(b_k)J_{n-1}(b_0)}\lesssim \frac{b_k-b_0}{n} \leq b_k-b_0. $$
From this estimate and (\ref{I_0}) we have (\ref{C2bd}).

\begin{remark}  Theorem \ref{th_st1} can be extended to the case that $q_0$ is null. In this case we take in (\ref{b_k}) $k(b_0)=0$ and from (\ref{potencial_nulo}) 
$$\| \left(\Lambda_{q_0}-\Lambda_{q_k}\right)g\|^2_{ H^{-1/2}_{\#}}   \leq   \left|\frac{b_kJ_1(b_k)}{b_kJ_1(b_k)\log b_k+J_0(b_k)}\right|^2 \left| g_0
	\right|^2$$
	$$+ \sum_{n=1}^\infty  \left|  1- \frac{J_{n-1}(b_k)-b_k^{2n}J_{n+1}(b_k)}{J_{n-1}(b_k)+b_k^{2n}J_{n+1}(b_k)}   \right|^2 (1+n^2)^{1/2}\left(  \left|g_n   \right|^2 +\left|g_{-n}^s   \right|^2  \right) ,$$
	by using $J_1(r) \leq \frac{r}{2}$ $r \in (0,1)$, $b_k \in (0,1)$,  (\ref{decreciente}) and (\ref{estimacion_inferior})
	$$ \lesssim b_k^4 \left|g_0\right|^2+ \sum_{n=1}^\infty    \frac{b_k^{4n}J_{n+1}^2(b_k)}{J_{n-1}^2(b_k)}   (1+n^2)^{1/2}\left(  \left|g_n   \right|^2 +\left|g_{-n}^s   \right|^2  \right) ,$$
$$ \lesssim b_k^4 \left|g_0\right|^2+ \sum_{n=1}^\infty    \frac{b_k^{2n+4}}{n(n+1)} (1+n^2)^{1/2}\left(  \left|g_n   \right|^2 +\left|g_{-n}^s   \right|^2  \right)  \lesssim  \frac{1}{k^4}  \| g\|^2_{ H^{1/2}_{\#}}.$$
\end{remark}

\end{proof}

\begin{proof} (of Theorem \ref{th 0.1})

Let $q_1(x)=\gamma_1 \chi_{B(0,b_1)}(x), \; q_2(x)=\gamma_2 \chi_{B(0,b_2)}(x)$ in  $
F_b$ and $g(\theta)=\frac{1}{2^{1/4}}e^{i\theta}$.

$$
\| \Lambda_{q_1}-\Lambda_{q_2} \|^2_{\mathcal{L}(H^{1/2}_{\#}; H^{-1/2}_{\#})} 
 \geq \| \left(\Lambda_{q_1}-\Lambda_{q_2} \right)g \|_{H^{-1/2}_{\#} }^2 
$$
\begin{equation}\label{regina_1}	
	 = \left|  \frac{J_{0}(b_1\sqrt{\gamma_1})-b_1^{2}J_{2}(b_1\sqrt{\gamma_1})}{J_{0}(b_1\sqrt{\gamma_1})+b_1^{2}J_{2}(b_1\sqrt{\gamma_1})}- \frac{J_{0}(b_2\sqrt{\gamma_2})-b_2^{2}J_{2}(b_2\sqrt{\gamma_2})}{J_{0}(b_2\sqrt{\gamma_2})+b_2^{2}J_{2}(b_2\sqrt{\gamma})}   \right|^2  
	 \end{equation}
$$
\geq \frac{4\textbf{II}^2}{\left(1+\frac{b_1^4 \gamma_1}{8}     \right)^2\left(1+\frac{b_2^4 \gamma_2}{8} \right)^2},
$$
where
$$\textbf{II}=\left| b_2^2 J_0(b_1 \sqrt{\gamma_1})  J_2(b_2 \sqrt{\gamma_2})- b_1^2 J_0(b_2 \sqrt{\gamma_2})  J_2(b_1 \sqrt{\gamma_1})     \right|,$$ and we have used (\ref{estimacion_inferior}) for $n=0,2$. On the other hand, 
\begin{equation}\label{tres}
\textbf{II}\geq \frac{1}{8}\left| b_2^4\gamma_2-  b_1^4\gamma_1     \right|-\textbf{J}_1 -\textbf{J}_2-\textbf{J}_3,
\end{equation}
where
\begin{equation}\label{jota_1}
\textbf{J}_1=\left| b_2^2S_2(b_2 \sqrt{\gamma_2}) - b_1^2S_2(b_1 \sqrt{\gamma_1})     \right|,
\end{equation}
\begin{equation}\label{jota_2}
\textbf{J}_2=\frac{1}{8}\left|   b_2^4\gamma_2S_0(b_1 \sqrt{\gamma_1}) -b_1^4\gamma_1S_0(b_2 \sqrt{\gamma_2})    \right|,
\end{equation}
and
\begin{equation}\label{jota_3}
\textbf{J}_3=\left|   b_2^2S_0(b_1 \sqrt{\gamma_1}) S_2(b_2 \sqrt{\gamma_2}) -  b_1^2S_0(b_2 \sqrt{\gamma_2}) S_2(b_{1} \sqrt{\gamma_{1}})     \right|.
\end{equation}

To estimate  $\textbf{J}_i$, $i=1,2,3$, we use  (\ref{estimacion_inferior}),  (\ref{bessel_orden_0}), (\ref{bessel_orden_2}) and  Lemma     \ref{a}. We get

\begin{equation}\label{jota_1b}
\textbf{J}_1 \leq \frac{b_2^4 \gamma_2^2\left| b_1^2-b_2^2  \right|}{30 \pi} +\frac{b_1^2\left(b_2^2 \gamma_2+b_1^2 \gamma_1 \right)\left| b_2^2 \gamma_2 - b_1^2 \gamma_1   \right|}{96}.
\end{equation}

\begin{equation}\label{jota_2b}
\textbf{J}_2 \leq \frac{b_1^2 \gamma_1\left| b_2^4 \gamma_2-b_1^4 \gamma_1   \right|}{32} +\frac{b_1^4 \gamma_1 \left|  b_2^2 \gamma_2-  b_1^2 \gamma_1 \right|}{32}.
\end{equation}

\begin{equation}\label{jota_3b}
\textbf{J}_3 \leq  \frac{b_1^2b_2^4\gamma_1 \gamma_2^2 \left| b_2^2-b_1^2   \right|}{120 \pi} +\frac{b_1^6 \gamma_1^2 \left| b_1^2 \gamma_1-  b_2^2 \gamma_2 \right|}{36 \pi^\frac{3}{2}}   +\frac{b_1^4b_2^4 \gamma_1\gamma_2 \left| b_1^2 \gamma_1-  b_2^2 \gamma_2 \right|}{36 \pi^\frac{3}{2}}   
\end{equation}
$$+\frac{b_1^2b_2^4 \gamma_1\gamma_2^2 \left| b_1^2 \gamma_1-  b_2^2 \gamma_2 \right|}{480 \pi^\frac{3}{2}}    .$$

\begin{description}

\item \underline{Proof of (\ref{estimacion_estabilidad})}.

      We  suppose  that  $b_1=b_2=b>0$. We obtain
$$
\begin{array}{lll}
\textbf{J}_1 \leq \frac{b^6}{96}\left| \gamma_1 -\gamma_2  \right| \leq 0.01041b^4 \| q_1-q_2  \|_{L^\infty (B(0,1)}, \vspace{0.2cm} \\ \textbf{J}_2 \leq \left( \frac{b^6}{32}+ \frac{b^6}{32}    \right)\left| \gamma_1 -\gamma_2  \right| \leq 0.0625 b^4 \| q_1-q_2  \|_{L^\infty (B(0,1)}, \vspace{0.2cm} \\ \textbf{J}_3 \leq  \left( \frac{b^8}{36 \pi^{\frac{3}{2}}} +\frac{b^{10}}{36 \pi^{\frac{3}{2}}}+\frac{b^8}{480 \pi^{\frac{3}{2}}} \right)\left| \gamma_1 -\gamma_2  \right|  \leq 0.01004 b^4 \| q_1-q_2  \|_{L^\infty (B(0,1))},
\end{array}
$$
and from (\ref{regina_1}) and (\ref{tres})
$$\| \Lambda_{q_1}-\Lambda_{q_2} \|_{\mathcal{L}(H^{1/2}_{\#}; H^{-1/2}_{\#})} \geq 0.04205 b^4 \| q_1-q_2  \|_{L^\infty },$$ that implies (\ref{estimacion_estabilidad}).

\item \underline{Proof of (\ref{estimacion_estabilidad_2})}.

Now $\gamma_1 = \gamma_2$. Let us define $$M(\gamma , b_1,b_2)=\gamma \left( b_1^3+b_1^2b_2+b_1b_2^2+b_2^3 \right).$$
It is easy to check that
$$
\begin{array}{llll}
\frac{1}{8}\left| b_2^4 \gamma_2 -b_1^4 \gamma_1   \right|=\frac{1}{8}M(\gamma , b_1,b_2) \left| b_2-b_1 \right|, \vspace{0.2cm} \\ \textbf{J}_1 \leq \left( \frac{1}{30 \pi}+ \frac{1}{9  \pi^{\frac{3}{2}}}   \right)M(\gamma , b_1,b_2) \left| b_2-b_1 \right|, \vspace{0.2cm} \\  \textbf{J}_2 \leq \left( \frac{1}{32}+ \frac{1}{256  \pi^{\frac{1}{2}}}   \right)M(\gamma , b_1,b_2) \left| b_2-b_1 \right|, \vspace{0.2cm} \\  \textbf{J}_3 \leq \left( \frac{1}{120}+ \frac{1}{18 \pi^{\frac{3}{2}}} + \frac{1}{420 \pi^{\frac{3}{2}}}   \right)M(\gamma , b_1,b_2) \left| b_2-b_1 \right|,
\end{array}
$$
therefore,
$$\| \Lambda_{q_1}-\Lambda_{q_2} \|_{\mathcal{L}(H^{1/2}_{\#}; H^{-1/2}_{\#})} \geq  \frac{2 }{\left(  1+\frac{1}{8} \right)^2}\left( \frac{\gamma}{8}\left|  b_1^4 - b_2^4    \right|   -\textbf{J}_1-\textbf{J}_2-\textbf{J}_3 \right)$$
$$\geq\frac{0.04216}{\left(  1+\frac{1}{8} \right)^2}M(\gamma , b_1,b_2) \left| b_2-b_1 \right|.$$ 
and we obtain (\ref{estimacion_estabilidad_2})

\end{description}

\end{proof}


\begin{thebibliography}{10}

\bibitem{A}
	 G. Alessandrini, 	\textit{Stable determination of conductivity by boundary mea- surements}, 	 
Appl. Anal. 27 (1988), no. 1-3, 153–172.




	  
	\bibitem{BHQ}  
	  E. Beretta, M. V. De Hoop, and L. Qiu, \textit{Lipschitz stability of an inverse boundary value
problem for a Schr\"{o}dinger-type equation}, SIAM Journal on Mathematical Analysis, 45
(2013), pp. 679–699
	
	
\bibitem{BIY}	
E. Bl\.{a}sten, O. Yu. Imanuvilov and M. Yamamoto, \textit{Stability and uniqueness for a two-dimensional inverse boundary value problem for less regular potentials}. Inverse Problems and Imaging 9(3) (2015), 709–723.	
	
	
	
	\bibitem{Gr}
	L. Grafacos;,
	\textit{Classical Fourier analysis},
	2th edn. Springer, (2008).
	
	\bibitem{I}
	D. V. Ingerman, 	\textit{Discrete and continuous Dirichlet-to-Neumann maps in the layered case}, SIAM J. Math. Anal., Vol 31, nº 6, (2000), 1214-1234.
	
	
	
	\bibitem{LE}
	N. N. Lebedev,
	\textit{Special functions and their applications},
	Dover Publications, New York Inc. (1972).
	
	
\bibitem{T1}
	J. Tejero, 	\textit{Reconstruction and stability for piecewise smooth potentials in the plane}, https://arxiv.org/pdf/1606.03020.pdf

\bibitem{T2}
	J. Tejero, 	\textit{Reconstruction of rough potentials in the plane}, https://arxiv.org/abs/1811.09481

\bibitem{Ul}	G. Uhlmann, \textit{Inverse problems: seeing the unseen}, Bull. Math. Sci., Vol 4, nº 2, (2014), 209-279.

	

	
\end{thebibliography}
\end{document}